\documentclass[a4paper]{amsart}
\usepackage{verbatim}
\usepackage[notcite,notref, final]{showkeys} %
\usepackage[dvips]{color}
\usepackage{amssymb}

\usepackage{lineno}
\usepackage{enumerate}
\usepackage[draft]{fixme}
\usepackage[active]{srcltx}

\usepackage{pstricks}
\usepackage{pst-node}
\usepackage{pst-plot}

\usepackage{float}

\usepackage{draftwatermark}

\usepackage{latexsym}
\usepackage{color}
\usepackage{amsmath}
 \usepackage{float}
\usepackage{mathrsfs} 

\floatstyle{plain} \restylefloat{figure} \restylefloat{table}

\allowdisplaybreaks

\newtheorem{theorem}{Theorem}[section]
\newtheorem{proposition}{Proposition}[section]
\newtheorem{lemma}{Lemma}[section]

\theoremstyle{definition}

\newtheorem{remark}{Remark}[section]

\floatstyle{ruled} \restylefloat{figure} \restylefloat{table}

\numberwithin{equation}{section}

\newdimen\vintbar
\def\vint{-\kern-\vintbar\int}

\def\A{\mathcal A}

\def\E{\mathbb E}
\def\e{\epsilon}
\def\F{\mathcal F}

\def\I{\mathcal I}

\def\L{\mathcal L}

\def\P{\mathbb P}
\def\Q{\mathcal Q}
\def\R{\mathbb R}

\def\0{\boldsymbol 0}

\renewcommand{\l}{\left}
\renewcommand{\r}{\right}

\begin{document}

\modulolinenumbers[1]

\SetWatermarkText{} 
\SetWatermarkFontSize{30cm}
\SetWatermarkScale{4}
\SetWatermarkLightness{0.8}

\title[Optimal switching under incomplete information]{A Brownian optimal switching problem\\under incomplete information}

\address{Marcus Olofsson\\Department of Mathematics, Uppsala University\\
S-751 06 Uppsala, Sweden}
\email{marcus.olofsson@math.uu.se}

\begin{abstract}
\medskip In this paper we study an incomplete information optimal switching problem in which the manager only has access to noisy observations of the underlying Brownian motion $\{W_t\}_{t \geq 0}$. The manager can, at a fixed cost, switch between having the production facility open or closed and must find the optimal management strategy using only the noisy observations. Using the theory of linear stochastic filtering, we reduce the incomplete information problem to a full information problem, show that the value function is non-decreasing with the amount of  information available, and that the value function of the incomplete information problem converges to the value function of the corresponding full information problem as the noise in the observed process tends to $0$.

\end{abstract}

\author{Marcus Olofsson}\thanks{This work was partially financed
by Jan Wallanders och Tom Hedelius Stiftelse samt Tore Browaldhs Stiftelse
through the project {\it Optimal switching problems and their applications in economics and finance}, P2010-0033:1.}

\maketitle

\section{Introduction}

In this paper we study the two mode optimal switching problem (OSP) under incomplete information. The manager can choose to be either in state $0$ or state $1$. In state $1$ the production facility is open and one continuosly collects the stoachstic revenue $X_t$. In state $0$, the facility is closed and generates no revenue.
The task of the manager is to choose when to open/close the facility in order to generate maximum profit. If the process $X_t$ is fully observed and there is no cost of switching between 'open' and 'closed', it is intuitevly clear that the manager should let the facility be open when $X_t \geq 0$ and closed when $X_t < 0$. However, when costs are associated to opening or closing, it may be more profitable to leave the plant open for a while even if $X_t <0$ (or, simliarly, leave it closed even if  $X_t \geq 0$). Furthermore, if the process $X_t$ cannot be fully observed, the manager has to base her decisions on incomplete information of the underlying process $X_t$. 

The problem of optimal switching under full information has been widely studied during the last decades and mainly three different approaches have turned out to be useful, two based on stochastic techniques and one on deterministic ditto. Firstly, Snell envelopes was used by \cite{DHP10} in combination with a verification theorem to prove the existence of a unique value function. 
Secondly, the optimal switching problem can be tackled using reflected backward stochastic differential equations (BSDEs) and we refer to \cite{DHP10}, \cite{HT07}, \cite{HZ10}, and the references therein for more on this reflected BSDE approach. For the connection between Snell envelopes and BSDEs we refer to \cite{EKPPQ97}. Thirdly, the optimal switching problem has been studied using deterministic methods based on systems of variational inequalities. More on this deterministic approach to the OSP can be found in \cite{AH09}, \cite{HM12}, \cite{LuNO12}, \cite{TY93} and the references therein. When dealing with the incomplete information optimal switching problem (IIOSP) considerably less work seems to be done and the author is only aware of the paper \cite{LiNO14a} 
, in which the authors formulate and develop numerical methods for this type of problem. Under incomplete information and when using the formalism of stochastic filtering as in \cite{LiNO14a}
, the IIOSP is in general not analytically tractable and numerical methods seem inevitable. However, in this paper we derive some analytical results concerning a simplistic model of the IIOSP introduced in \cite{LiNO14a}. To be more specific, we study the IIOSP in a setting similar to that of \cite{F78}, in which the related problem of incomplete information optimal stopping is studied. We prove that the value function is convex, non-decreasing in the amount of available information and that it converges to the value function of the standard OSP as the noise in the observation tends to $0$. 
We emphasize that, in the current situation, the analytical approach of this paper is mainly conceptual while the numerical method developed in \cite{LiNO14a} 
deal with more general IIOSPs and is readily available for applications.

The rest of this paper is organized as follows. 
In section \ref{sec:assandnot} we formulate the IIOSP, state precise assumptions and present some preliminaries. Section \ref{sec:mainresults} contains the main results while section \ref{sec:proofs} is devoted to proving these. In section \ref{sec:numerical} we give a numerical example of the problem studied in this paper. Finally, we end with section \ref{sec:conclusions} containing conclusions and some future lines of research.

\setcounter{equation}{0} \setcounter{theorem}{0} \setcounter{definition}{0}
\section{Problem formulation and preliminaries} \label{sec:assandnot}
\subsection{Full information optimal switching problems}
We begin by briefly outlining the standard OSP under full information. Let $X=\{X_t\}_{t \geq 0}$ be an $m$-dimensional stochastic process 
\begin{equation*}
dX_t = b(X_t) dt + \sigma(X_t) dW_t \hspace{1cm} X_t=x
\end{equation*}
and denote the infinitesimal generator of $X$ by $\L$. 
Let $\Q:=\{0,1,\dots, d\}$ be the finite set of available states. A management strategy is a combination of a non-decreasing sequence of stopping times $\{\tau_k\}_{k\geq 0}$, where, at time $\tau_k$, the manager decides to switch production from its current mode to another one, and a sequence of indicators $\{\gamma_k\}_{k\geq 0}$, taking values in $\Q$, indicating the mode to which the production is switched. At $\tau_k$  the production is switched from mode $\gamma_{k-1}$ to $\gamma_k$.  The cost of switching from state $i$ to state $j$ at time $s$ is denoted $c_{ij}(s,X_s)$. A strategy $(\{\tau_k\}_{k\geq 0},\{\gamma_k\}_{k\geq 0})$ can be represented by the simple function
$$
\mu_s=\sum_{i\geq 1} \gamma_i \chi_{(\tau_i, \tau_{i+1} ]}(s) + \gamma_0 \chi_{[\tau_0, \tau_1 ]}(s)
$$
indicating the current state of the facility. We will throughout the paper alternate between these two notations without further notice. If the facility is in state $i$ at time $s$ the genrated revenue per unit time is $\psi_i(s,X_s)$. Hence, when the production is run under a strategy $\mu$ over a finite horizon $[t,T]$, the expected total profit is
\begin{eqnarray*} 
J(t,x,\mu)=E\biggl [\biggl(\int_t^T\psi_{\mu_s}(s,X_s)ds-\sum_{t\leq \tau_k \leq T}c_{\gamma_{{k-1}},\gamma_{k}}(\tau_{k},X_{\tau_{k}})\biggr )\biggr ].
\end{eqnarray*}
The task in the OSP is to find the value function 
\begin{equation}\label{eq:FIproblem}
v(t,x)= \sup_{\mu \in \A^X_{t,i}} J(t,x,\mu)
\end{equation}
where $\A^X_{t,i}$ denotes the set of strategies adapted to the filtration generated by $X$ which are in state $i$ at time $t$. 

Under sufficient regularity conditions on the payoff functions $\psi_i$ and switching costs $c_{ij}$ and the so called ``no-loop condition'' the following theorems can be proven, see \cite{DHP10}. 
\begin{theorem} \label{thm:varineq}
The vector of value functions $(v_1(t,x), \dots, v_d(t,x))$ solves the system of variational inequalities
\begin{align*}
&\min \bigg \{v_i  - \max_{i \neq j} \l( v_j(t,x) - c_{ij}(t,x) \r), -\partial_t v_i(t,x) - \L v_i(t,x) - \psi_i(t,x) \bigg \} =0, \\
&v_i(T,x)= 0,
\end{align*}
in the viscosity sense.\footnote{For the definition of viscosity solutions see \cite{CIL92}.} Furthermore, $(v_1(t,x), \dots, v_d(t,x))$ is the unique solution satisfying the polynomial growth condition
$$
v_i(t,x) \leq C(1+|x|^\eta), 
$$
$i\in \Q$, for some $\eta \geq 1$.
\end{theorem}

\begin{theorem} \label{thm:finstrategy} 
There exists a finite strategy $\mu^\ast \in \A^X_{t,i}$ such that $J(t,x,\mu^\ast) \geq J(t,x,\mu)$ for any $\mu \in \A^X_{t,i}$.
\end{theorem}

For future reference we introduce the regions $C_i$ and $S_i$, $i\in \Q$, defined as
\begin{align*}
C_i &= \big \{ (t,x) \in [0,T] \times \R : v_i(t,x) > v_j(t,x) - c_{ij}(t,x) \big\}, &\notag \\
S_i &= \big \{ (t,x) \in [0,T] \times \R : v_i(t,x) =\max _{j \in \Q \setminus \{i \}} \l \{ v_j(t,x) - c_{ij}(t,x) \r \} \big\}. &
\end{align*}
The sets $S_i$ and $C_i$ are usually refered to as ``Switching regions'' and ``Continuation regions'', respectively. It is optimal to switch from region $i$ when the underlying process $(t,X_t)$ hits the switching region $S_i$, see \cite{DHP10}.

\subsection{Incomplete information optimal switching problems}
In contrast to the OSP outlined above, the manager in an IIOSP only has access to incomplete information about the underlying process $X$, information acquired through a fully observable $X$-dependent process $Y$. In particular, the manager can only observe the process  $Y$, solution to the stochastic differential equation
\begin{equation*}
dY_s = h(X_s)ds+dU_s, \hspace{1cm}Y_{t}=y,
\end{equation*}
where $U$ is a Brownian motion independent of $W$, and she must base her decisions solely on the information contained in $\F^Y$, the $\sigma$-algebra generated by $Y$. This lack of information prevents us from directly applying results concerning standard OSP.

 In the theory of stochastic filtering, and we refer to \cite{BC09} for details concerning this, the main goal is to compute conditional expectations $\mathbb E\left[\phi(X_t)|\F^{Y}_{t}\right]$ for suitably chosen test functions $\phi$. The solution to the stochastic filtering problem is the distribution of the random variable $X_t$ conditional on the $\sigma$-algebra $\F^{Y}_t$. We let $\pi_t$ denote this distribution and hence
\begin{equation}\label{eq:pi}
\mathbb E\left[\phi(X_t)|\F^Y_t\right] = \int_{\mathbb{R}^{m}}\phi(x)\pi_t(dx) \triangleq \pi_t(\phi).
\end{equation}
Based on the solution $\pi_t$ of the stochastic filtering problem we define, following \cite{LiNO14a}, the expected total profit when the production is run under an $\F^Y_t$-adapted strategy $\mu^Y=(\{\tau^Y_k\}_{k\geq 0},\{\xi^Y_k\}_{k\geq 0})$, over a finite horizon $[0,T]$, to be
\begin{equation*}
\tilde J(\mu^Y)=\mathbb E\biggl [\biggl(\int_0^T\mathbb E\bigl[\psi_{\mu^Y_s}(s,X_s)|\F^{Y}_{s}\bigr]ds-\sum_{k\geq 1}\mathbb E\bigl[c_{\xi^Y_{{k-1}},\xi^Y_{k}}( \tau^Y_{k},X_{\tau^Y_k})|\F^{Y}_{\tau^Y_{k}}\bigr]\biggr )\biggr ].
\end{equation*}
Let for $t\in[0,T]$, $i\in \Q$, $\mathcal A_{t,i}^Y$ denote set of $\F_t^Y$-adapted strategies such that $\tau_1 \geq t$ and $\xi_0=i$ a.s. and recall the notation introduced in \eqref{eq:pi}. Given $t \in [0,T]$ and a probability measure $\hat \pi$, we define $v_i(t, \hat \pi)$, the value function associated with the IIOSP, to equal
\begin{equation}\label{eq:valuefcnpartial}
\sup_{\mu\in \mathcal A^Y_{t,i}} \E \l [\int _t ^T \pi_s \l ( \psi_{\mu_s}(s,X_s) \r ) ds - \sum _{ t \leq \tau_n \leq T} \pi_{\tau_k} \l
 ( c_{\xi_{n-1}\xi_n} (\tau_k, X_{\tau_k}) \r ) \, \vline \, \pi_t=\hat \pi \right ].
\end{equation}
The function $v_i(t, \hat \pi)$ stands for the optimal expected profit if, at time $t$,
the production is in mode $i$ and the probability distribution of the unknown $X_t$ is $\hat \pi$. 

\subsection{Assumptions}
We now turn to the specific setup studied in this paper.  We let $(\Omega, \F, (\F_t)_{t\geq 0}, \P)$ be a complete filtered probability space on which we define two independent Brownian motions $(W_t, \F_t)$ and $(\bar W_t, \F_t)$. The notation $W^{x}$ is used to indicate that the Brownian motion is started from the point $x$ at time $t=0$. The assumption that the initial observation is made at time $t=0$ is made w.l.o.g., see Remark 3.1 of \cite{LiNO14a}. We let $W^x$ be the underlying stochastic process and let, for an arbitrary but fixed $\e > 0$, the process $\xi^\e$ defined by,
\begin{equation*} 
d\xi^\e_t = W^x_t dt + \e d \bar W_t, \hspace{1cm} \xi_0^\e =x,
\end{equation*}  
represents noisy observations of $W$. We denote by $\{\F^\e_t\}_{t\geq 0}$ the filtration generated by $\xi^\e$. We stress that $\F_t \not \subset \F^\e_t$ and consequently the value of $W^x_t$ is not known based on the information in $\F^\e_t$. The manager can choose between having the production facility 'open' (state $1$) or 'closed' (state $0$), i.e., $\Q=\{0,1\}$, and the corresponding payoff functions are 
\begin{eqnarray} \label{eq:payoffs}
\psi_0(x)=0, \hspace{1cm}
\psi_1(x)=x. 
\end{eqnarray}
The manager can only observe the process $\xi^\e$ and the decision to open/close the production at time $t$ must hence be made based solely on the information contained in $\F^\e_t$. Concerning the cost of switching we assume that
\begin{equation}\label{eq:triangleswitching}
\mbox{$c_{01}$ and $c_{10}$ are positive fixed constants.}
\end{equation}

The solution to the IIOSP outlined above is the value function $v_i^\e(t,\hat \pi)$, $i \in \Q$, defined as 
\begin{equation} \label{eq:PIproblem}
v_i^\e(t, \hat \pi)= \sup _{\mu \in A^{\xi^\e}_{t,i}}  \E \l [ \int _t ^T \E \l[W^{x}_s \I_{\{\mu_s =1\}} \,\vline \, \F^\e_s \r]  ds - \sum _{n \geq 1} c_{\mu_{\tau_{n-1}} \mu_{\tau_n}}  \,\vline \,  W^{x}_t \sim \hat \pi \r ],
\end{equation}
where $\hat \pi$ is a probability measure and $\A^{\xi^\e}_{t,i}$ denotes the set of strategies adapted to the filtratrion  $\F^\e_t$ which are in state $i$ at time $t$. We denote by $m^{x,\e}_t= \E \l [ W^x_t \, \vline \, \F^\e_t  \r]$ and $\theta^{x,\e}_t = \E \l [ (m^{\e}_t - W^x_t)^2  \r ]$ the conditional mean and conditional variance of $W^x_t$. Using this notation and the fact that $\mu$ is by definition an $\F^\e$-measurable strategy \eqref{eq:PIproblem} can be simplified to read
\begin{equation*} 
v_i^\e(t, \hat \pi)= \sup _{\mu \in A^{\xi^\e}_{t,i}}  \E \l [ \int _t ^T m^{x,\e}_s  \I_{\{\mu_s =1\}}ds - \sum _{n \geq 1} c_{\mu_{\tau_{n-1}} \mu_{\tau_n}}  \,\vline \, W^x_t \sim \hat \pi \r ]
\end{equation*}
As with \eqref{eq:valuefcnpartial}, the function $v^\e_i(t, \hat \pi)$ can be interpreted as the optimal expected profit if, at time $t$, the facility is in state $i$ and, given the available observations, the probability distribution of $W^x_t$ is $\hat \pi$. For future reference we also introduce the value function $v_i^0(t, \hat x)$ defined as
\begin{equation*} 
v_i^0(t, \hat x)= \sup _{\mu \in A^{W^x}_{t,i}}  \E \l [ \int _t ^T W^x_s  \I_{\{\mu_s =1\}}ds - \sum _{n \geq 1} c_{\mu_{\tau_{n-1}} \mu_{\tau_n}}  \,\vline \, W^x_t=\hat x \r ],
\end{equation*}
i.e., the value function of the standard OSP corresponding to the assumptions above.

\section{Main results}\label{sec:mainresults}
To simplify the statements of the main results, we first state an immidiate consequence of Theorem \ref{thm:filtering}, further explained in the bulk of the paper. In particular, Proposition \ref{prop:dimreduction} reduces the apriori infinite dimensionality of the IIOSP.

\begin{proposition}\label{prop:dimreduction}
The stochastic probability measure $\pi_t$ of $W^x_t$ conditional on $\F^{\e}_t$ is fully characterized by its conditional mean $m_t^{x,\e}=\E\l [W^x_t \, \vline \, \F^\e_t \r]$ and a time-dependent deterministic function $\theta_t$.
\end{proposition}
With this proposition in mind, it is clear that the value function $v^\e(t,\hat \pi)$ can be expressed as a function $v^\e(t,m): [0,T]\times \R \to \R$. With slight abuse of notation we will write
\begin{equation} \label{eq:PIproblem3}
v^\e_i(t, \hat \pi) = v_i^\e(t, m)= \sup _{\mu \in A^{\xi^\e}_{t,i}}  \E \l [ \int _t ^T m^{x,\e}_s  \I_{\{\mu_s =1\}}ds - \sum _{n \geq 1} c_{\mu_{\tau_{n-1}} \mu_{\tau_n}}  \,\vline \,  m^{x,\e}_t=m  \r ]
\end{equation}
for $m = \int_{\R} x d \hat \pi(dx) = \E \l[ W^x_t \, \vline \, \F^\e_t \r]$. We are now ready to state the results. 
\begin{theorem}[Monotonicity in $\e$] \label{prop:monotone}
For any $\e_1, \e_2$ such that $0 \leq \e_1 \leq \e_2$ we have
\begin{equation*}
 v_i^{\e_2}(t,m) \leq v_i^{\e_1}(t,m).
\end{equation*}
\end{theorem}
In words, Theorem \ref{prop:monotone} states that the value of information is positive, i.e., decreasing the noise in the observation process $\xi$ increases the optimal expected profit. The intuitive reason for this result is that more accurate observations simplify the managers task of optimally controlling the production facility. When the noise in $\xi$ tends to $0$, more and more information becomes available and the task of the manager starts to resemble that under full information. Consequently, the expected optimal profit under incomplete information should tend to the expected optimal profit under full information. This intuition is confirmed by the follwoing theorem.
\begin{theorem}[Convergence of $v^\e$ to $v^0$] \label{thm:convergence}
As $\e \searrow 0$ the value function $v^\e(t,m) \to v^0(t,m)$ and 
\begin{equation*}
0 \leq v^0(t, m) - v^\e(t,m) \leq  (T-t) \l( \frac{4}{\pi} \log (2) -\frac{2}{\pi} \r) ^{1/2} \sqrt{\e}
\end{equation*}
\end{theorem}

%

%
%
%

\section{Proof of the main results} \label{sec:proofs}
The outline of this section is as follows. Firstly, we state some results from the theory of stochastic filtering which will constitute a base for the proofs of the results in Section \ref{sec:mainresults}. We then show that the IIOSP stated above can be reduced to a full information OSP and apply the standard theory, in particular the results presented in Section \ref{sec:assandnot}, to prove
Theorem \ref{prop:monotone} and Theorem \ref{thm:convergence}.

\subsection{Stochastic filtering}
Results concerning standard OSP are in general not directly applicable to the IIOSP due to the lack of information. 
The main idea of this paper is to reduce the IIOSP to a standard OSP to which the results of Section \ref{sec:assandnot} can be applied. This reduction is made possible by the following well-known results from the theory of linear stochastic filtering. Proofs of the below statements can be found in \cite{LS78} (Theorem 10.3 and Theorems 7.12 and 7.16).
\begin{theorem} \label{thm:filtering}
The distribution of $W^x_t$  conditional on $\F^\e_t$ is Gaussian. Furthermore, the conditional mean $m^{x,\e}_t= \E \l [ W^x_t \, \vline \, \F^\e_t  \r]$ and variance $\theta^\e_t = \E \l [ (m^{x,\e}_t - W_t)^2  \r]$ are the unique pair of processes satisfying
\begin{align} \label{eq:mean}
&dm^{x,\e}_t = \frac {\theta^\e_t  }{\e^2}\l( d\xi^\e_t - m^{x,\e}_t dt \r), \qquad m^\e_0=x, \\
\label{eq:variance}
&\frac {d \theta}{dt} = 1 - \l ( \frac{\theta^\e_t}{\e}\r)^2, \qquad   \theta^\e_0=0.
\end{align}
\end{theorem} 
\begin{theorem} \label{thm:filtering2}
The random process $N_t = (N_t, \F^{\e}_t)$, $0 \leq t \leq T$, with
$$
N_t=\xi^\e_t - \int _0^t m_s^\e ds,
$$
is a Wiener process and the filtration $\F^N$ generated by $N$ conincides with that generated by $\xi^\e$, i.e., $\F^{\e}_t = \F_t^N$  for all $t \in [0,T]$.
\end{theorem} 

Note that the starting point $m_0=\E [W^x_0] $ is given by the starting point of the underlying process $W^x$. Since the distribution of $W^x_t$  conditional on $\F^\e_t$ is Gaussian it is fully characterized by its mean $m_t$ and variance $\theta_t$. Proposition~\ref{prop:dimreduction} follows immediately since $\theta_t$ is given by the ordinary differential equation \eqref{eq:variance}.

\subsection{Reduction to the complete information case}
Before proving the main results, we reduce the problem to a complete information setting and prove the convexity of the value function.

\begin{proposition} \label{prop:altrepres}
Fix $\e >0$ and assume \eqref{eq:payoffs} and \eqref{eq:triangleswitching}. Let $X^{m,t,\e}_s$ be the solution to \begin{equation}\label{eq:Xdynamics}
dX^{m,t,\e}_s = \tanh(\dfrac{s}{\e} )dW_s, \hspace{0.5cm} X^{m,t,\e}_t = m,
\end{equation}
and let
\begin{equation*}
v_i(t,m) = \sup_{\mu \in \A^{X}_i }  \E \l [\int_t ^T X^{m,t,\e}_{s} \I_{\{\mu_s=1\}}ds - \sum _{n \geq 1} c_{\mu_{\tau_{n-1}} \mu_{\tau_n}} \r ]
\end{equation*}
be the value function of the full information OSP \eqref{eq:FIproblem}. Then, $v_i^\e(t, m) = v_i(t,m)$.
\end{proposition}

\begin{proof}
Recall that
\begin{equation*} 
v_i^\e(t, m)= \sup _{\mu \in A^\xi_{t,i}}  \E \l [ \int _t ^T m^{x,\e}_s  \I_{\{\mu_s =1\}}ds - \sum _{n \geq 1} c_{\mu_{\tau_{n-1}} \mu_{\tau_n}}  \,\vline \, m^{x,\e}_t =m  \r ]
\end{equation*}
where $m^{x,\e}_t= \E \l [W^x_t \, \vline \, \F^\e_t \r ]$. By Theorem \ref{thm:filtering} the dynamics of $m_t^{x,\e}$ is given by \eqref{eq:mean} which, after solving \eqref{eq:variance} and inserting the solution $\theta^\e_t =\e \tanh (\dfrac{t}{\e})$,  reads
\begin{eqnarray} \label{eq:processes12}
dm^\e_t 
&=&  \frac 1 \e \tanh (\dfrac {t}{\e}) d\xi^\e_t - \frac{1}{\e} \tanh (\dfrac {t}{\e})m^\e_t dt. \hspace{0.5cm} 
\end{eqnarray}
Since the value function \eqref{eq:PIproblem3} is an expected value we can replace the underlying process (and corresponding set of strategies) in our optimal switching problem with any other process having the same distributional properties without changing its value. With this in mind, we introduce the innovations process
\begin{equation}\label{eq:N_trelation}
N_t = \xi_t^\e - \int _0 ^t m^\e_s ds
\end{equation}
and put $dR_t = \frac 1 \e dN_t$. Then, by \eqref{eq:N_trelation} the process $\xi^\e$ admits a represenation of the form
\begin{equation*}
d\xi^\e_t = m^{x,\e}_t dt + \e dR_t.
\end{equation*}
and by Theorem \ref{thm:filtering2} the process $(R_t,\F_t^\e)$ is a Brownian motion. Furthermore, the $\sigma$-algebra $\F_t^R$ generated by $R_t$ coincides with $\F^\e_t$ for all $t \geq0$. 
The above together with \eqref{eq:processes12} and the initial condition $m^{x,\e}_t=m$ gives that
$$
m^{\e}_s = m +\int _t ^s  \tanh (\dfrac {r}{\e})d R_r.
$$
Recall the underlying Brownian motion $(W_t,\F_t)$ and the dynamics \eqref{eq:Xdynamics} of $X^{m,t,\e}$,
$$
X^{m,t,\e}_s = m + \int _t ^s \tanh (\dfrac r \e) dW_r.
$$
The distributional properties of the Brownian motions $(R_t, \F_t^\e)$ and $(W_t,\F_t)$ conincide and hence, since  $\F_t^R=\F^\e_t$ for all $t\geq 0$, it follows that
\begin{eqnarray}\label{eq:gammainserted}
v_i^\e(t,m)&=&  \sup_{\mu \in \A^\xi_i} \E \l [\int _t^T m^{x,\e}_s \I_{\{\mu_s =1\}} ds - \sum_{n\geq 1}  c_{\mu_{\tau_{n-1}} \mu_{\tau_{n}}} \, \vline \, m^{x,\e}_t =m \r] \notag \\
&=& \sup_{\mu \in \A^R_i} \E \l [\int _t^T m^{x,\e}_s \I_{\{\mu_s =1\}} ds - \sum_{n\geq 1}  c_{\mu_{\tau_{n-1}} \mu_{\tau_{n}}} \, \vline \, m^{x,\e}_t =m \r] \notag \\
&=& \sup_{\mu \in \A_i} \E \l [\int _t ^T X^{m,t,\e}_s\I_{\{\mu_s =1\}} ds -  \sum_{n\geq 1}   c_{\mu_{\tau_{n-1}} \mu_{\tau_{n}}} \r] = v_i(t,m)
\end{eqnarray}
and the proof is complete.
\end{proof}

\begin{remark}
Note that although the process $m^{x,\e}_t$ is not explicitly observable, its path is completely determined by the deterministic function $\theta_t$ and the observed process $\xi^\e$. Hence, we can w.l.o.g. consider $m^{x,\e}$ as being the observed process rather than $\xi^\e$. 
\end{remark}

\begin{lemma}\label{lemma:convexity}
For any $\e > 0$, the value function $v_i^\e(t,m)$, $i \in \Q$, is convex in $m$.
\end{lemma}
\begin{proof}
Recall the characterization of $v_i^\e(t,m)$ given in Proposition~\ref{prop:altrepres}.  Since the dynamics of $X^{m,t,\e}$ is independent of its current value, the resulting process is linear w.r.t. its starting point, i.e.,  $X^{m+ \eta,t,\e}= X^{m,t,\e}+\eta$. 
Let $\mu^\ast \in \A^X_i$ be an optimal strategy for $v_i^\e(t,m)$ so that
\begin{equation*}
v_i^\e(t,m) =  \E \l [ \int _t ^T X^{m,t,\e}_s \I_{\{ \mu^\ast_s =1\}}ds - \sum _{n \geq 1} c_{\mu^\ast_{\tau_{n-1}}  \mu^\ast_{\tau_n}} \r].
\end{equation*}
We now perturb the initial condition $m$ by $\eta$ and consider $v_i^\eta(t,m+\eta)$. The strategy $\mu^\ast$ is sub-optimal for $v^\e_i(t,m+\eta)$ and hence 
\begin{align}\label{eq:ineq1}
v_i^\e(t,m+\eta) \geq&  \E \l [ \int _t ^T  X^{m+\eta,t,\e}_s \I_{\{\mu^\ast_s =1\}}ds - \sum _{n \geq 1} c_{\mu^\ast_{\tau_{n-1}} \mu^\ast_{\tau_n}} \r] \notag \\
= &   \E \l [ \int _t ^T X^{m,t,\e}_s \I_{\{\mu^\ast_s =1\}}ds - \sum _{n \geq 1} c_{ \mu^\ast_{\tau_{n-1}}  \mu^\ast_{\tau_n}}  +\eta \int _t ^T  \I_{\{ \mu^\ast_s =1\}}ds \r] \notag \\
=& v^\e_i(t,m) + \eta f(t,m,i),
\end{align}
where $f(t,m,i)=\int _t ^T  \I_{\{ \mu^\ast_s =1\}}ds$ is the expected time spent in state $1$, using the optimal strategy for the starting point $(t,m,i)$. Repeating the above for the optimal strategy associated to the initial value $(t, m + \eta)$ yields
$$
v_i^\e(t,m) \geq v_i^\e(t,m+\eta) - \eta f(t,m+\eta,i) .
$$
Since $\int_t ^T \I_{\{\hat \mu_s =1\}}ds \geq 0$ it follows from \eqref{eq:ineq1} that $v_i^\e(t,m)$ is non-decreasing in $m$. Furthermore, combining the inequalities above we find
$$
v_i^\e(t,m) +  \eta f(t,m+\eta,i)  \geq v_i^\e(t,m+\eta) \geq v_i^\e(t,m) + \eta  f(t,m,i) ,
$$
i.e., the time spent online increases with the initial starting point. Dividing by $\eta$ and letting $\eta \to 0$ gives
\begin{align*}
f(t,m,i) \leq \lim_{\eta \searrow  0}\frac{ v^\e_i(t,m+\eta) -v^\e_i(t,m)}{\eta}  \leq  \lim_{\eta \searrow 0}f(t,m+\eta,i).
\end{align*}
Since $f(t,m,i)$ is non-decreasing in $m$, we can conclude that for any $\nu >0$
$$
D_x^+ v^\e_i(t,m) \leq \lim_{\eta \to 0} f(t,m + \eta) \leq f(t,m+\nu) \leq D_x^+ v_i^\e( t,m +\nu),
$$
where $D_x^+v_i^\e(t,m)$ denotes the right spatial derivative of $v_i^\e$ at $(t,m)$. We conclude that the right derivative of $v_i^\e$ is increasing at $m$. Convexity of $v_i^\e(t,m)$ now follows since $m$ was arbitrary. 
\end{proof}

\subsection{Proof of Propostition \ref{prop:monotone}} 
By combining Propostition \ref{prop:altrepres} and Theorem \ref{thm:varineq} we conclude that the vectors $(v_0^{\e_k},v_1^{\e_k})$, $k\in \{1,2\}$, solve, respectively, the systems of variational inequalities
\begin{align} \label{eq:system}
&\min \l \{ \varphi_0 - (\varphi_1 - c_{01}) , -\partial_t \varphi_0- \L^{\e_k} \varphi_0 - \psi_0 \r \} = 0 \notag \\
&\min \l \{ \varphi_1 - (\varphi_0 - c_{10}) , -\partial_t \varphi_1  -\L^{\e_k} \varphi_1 -\psi_1\r \} =0,
\end{align}
where 
$$\L^{\e_k} =\frac {1}{2} \tanh^2 (\dfrac{t}{\e_k})\partial_{xx}$$ 
is the generator of the process $X^{m,t,\e_k}$. We now intend to prove that $(v_0^{\e_2}, v_1^{\e_2})$ is a subsolution to the system above with $k=1$, i.e., that
\begin{align}\label{eq:system2}
(i)&\min \l \{ v^{\e_2}_0 - (v^{\e_2}_1 - c_{01}) ,  -\partial_t v^{\e_2}_0-\L^{\e_1} v^{\e_2}_0 - \psi_0 \r \} \leq0 \notag \\
(ii)&\min \l \{ v^{\e_2}_1- (v^{\e_2}_0 - c_{10}) ,  -\partial_t v^{\e_2}_1-  \L^{\e_1} v^{\e_2}_1 -\psi_1\r \} \leq0,
\end{align}
and then apply the comparison principle for \eqref{eq:system}. We focus on \eqref{eq:system2} $(i)$, the inequality in \eqref{eq:system2} $(ii)$ being treated similarly. Firstly, on the region $S_0^{\e_2}$ \eqref{eq:system2} $(i)$ is trivially satisfied since $v_0^{\e_2} = v_1^{\e_2} - c_{01}$ by definition. Hence, we only need to show that on the region $C_0^{\e_2}$, where $v^{\e_2}_0$ is above its obstacle by construction, we have 
\begin{equation}\label{eq:NTSineq}
-\partial_t v_0 ^{\e_2} -\L^{\e_1}v_0^{\e_2} - \psi_0 \leq 0,
\end{equation}
in the viscosity sense. Assume that $\varphi - v_0^{\e_2}$ has a local minimum at $(\hat t, \hat m) \in C_0^{\e_2}$. Since $v_0^{\e_2}$ is a viscosity solution to \eqref{eq:system} with $k=2$ we have
$$
-\partial_t \varphi(\hat t, \hat m) - \L^{\e_2} \varphi(\hat t, \hat m) - \psi_0(\hat t, \hat m) \leq 0.
$$
In particular, we see that $-\partial_t \varphi(\hat t, \hat m)  - \psi_0(\hat t, \hat m) \leq \L^{\e_2} \varphi(\hat t, \hat m)$.
Hence, at $(\hat t, \hat m)$ we have
\begin{align}\label{eq:ineq}
&-\partial _t \varphi(\hat t, \hat m)  -\L^{\e_1}\varphi(\hat t, \hat m)  - \psi_0(\hat t, \hat m) \notag \\
& \leq  \L^{\e_2} \varphi(\hat t, \hat m) -\L^{\e_1}\varphi(\hat t, \hat m) \notag \\
& = \frac 12\l (  \tanh^2(\dfrac{\hat t}{\e_2})  -\tanh^2(\dfrac{\hat t}{\e_1})  \r )\partial_{mm} \varphi(\hat t, \hat m).
\end{align}
The monotonicity of $\tanh^2(t)$ for $t \geq 0$ and the assumption $\e_1 \leq \e_2$ gives 
$$\frac 1 2\l(\tanh^2(\dfrac{\hat t}{\e_2})  -\tanh^2(\dfrac{\hat t}{\e_1})\r) \leq 0.$$ 
A function $\varphi$ is convex if and only if it is convex in the viscosity sense, see \cite{ALL97}, and hence it follows from Lemma \ref{lemma:convexity} that $\partial_{mm} \varphi(\hat t, \hat m) \geq0$. We can thus conclude from \eqref{eq:ineq} that \eqref{eq:NTSineq} holds.
This proves \eqref{eq:system2} $(i)$ since $S_0^{\e_2} \cup C_0^{\e_2} = \R$. Repeating the above arguments for $v_1^{\e_2}$ proves \eqref{eq:system2} $(ii)$ and thus the subsolution property is proven. Since $v_i^{\e_1}(T,x) = v_i^{\e_2}(T,x)$, $i \in \Q$, the theorem now follows from the comparison principle for \eqref{eq:system}. 

\subsection{Proof of Theorem \ref{thm:convergence}}
That $v_i^0(t,m)\geq v_i^\e(t,m)$ for any $\e >0$ is proven as Proposition \ref{prop:monotone} using $\tanh^2(t) \leq 1$ for any $t \geq 0$. We omit the details. By Theorem \ref{thm:finstrategy} there exists a finite optimal strategy $\mu^\ast \in \A^X_i$ for the full information problem \eqref{eq:FIproblem}. This strategy is sub-optimal in \eqref{eq:gammainserted} and hence
\begin{align} \label{eq:reads}
v_i^0 (t,m)- v_i^\e(t,m) =& \E \l [\int _t ^T W^{m,t}_s\I_{\{\mu^\ast_s =1\}} ds\r ]-  \E \l [\sum_{n\geq 1}   c_{\mu^\ast_{\tau_{n-1}} \mu^\ast_{\tau_{n}}} \r]\notag \\
& - \sup _{\mu \in \A^\xi_{i,t}} \E \l [\int _t ^T X^{m,t,\e}_s\I_{\{\mu_s =1\}} ds -\sum_{n\geq 1}   c_{\mu_{\tau_{n-1}} \mu_{\tau_{n}}} \r]  \notag \\
\leq& \E\l [\int _t ^T  \l( W^{m,t}_s - X^{m,t,\e}_s \r )\I_{\{\mu^\ast_s=1\}}  ds \r], 
\end{align}
where $W_s^{m,t}$, $s \geq t$ is the Brownian motion $W^x$ conditional on $W^x_t=m$.
Following \cite{F78} we introduce the notation 
\begin{align}
Q^\e_{t,s} =&\int_t ^s \l( 1-\tanh (\dfrac{u}{\e}) \r)dW_u \notag \\
N_{t,r}^\e =& \sup_{s \in [t,r]} Q^\e_{t,s}, \notag \\
N_{t,\infty}^\e =& \sup_{s \geq t} Q^\e_{t,s} \notag .
\end{align}
By definition 
$$W^{m,t}_s - X^{m,t,\e}_s =  \int _t ^s  \l(1 - \tanh(\dfrac{r}{\e})\r ) dW_r$$ 
and hence \eqref{eq:reads} yields
\begin{align}
&v_i^0 (t,m)- v_i^\e(t,m) \leq \E \l [ \int _t ^T  Q^\e_{t,s}  \I_{\{\mu^\ast_s=1\}} ds \r] \notag \\
 & \leq \E \l [ \sup _{s \in [t,T]} Q^\e_{t,s}   \int _t ^T \I_{\{\mu^\ast_s=1\}} ds \r]\leq  \E \l[ \sup _{s \geq t} Q^\e_{t,s}   \int _t ^T \I_{\{ \mu^\ast_s=1\}} ds \r] \notag \\
&\leq \E\l[  N^\e_{t,\infty}\r] (T-t).\notag
\end{align}
%
Since $\E \l[ N_{t,\infty}^\e \r]=  \E \l[ \lim_{r \to \infty} N^\e_{t,r} \r]$ and $N^\e_{t,r}$ is monotone in $r$ we may apply the monotone convergence theorem to find
$$
\E [N_{t,\infty} ^\e] =  \E \l[\lim_{r \to \infty} N^\e_{t,r} \r] = \lim_{r \to \infty} \E \l[ N^\e_{t,r} \r].
$$ 
By the reflection principle it follows that
$$
\E[ N^\e_{t,r}] 
= \l ( \frac{2}{\pi}  \int _t
^s \l(1- \tanh(\dfrac{r}{\e}) \r)^2 dr \r)^{1/2}.
$$
To be more explicit, applying the reflection principle to the stochastic process $Q^\e_{t,s}$  yields $\P(N_{t,r}^\e \geq b) = 2 \P(Q^\e_{t,r} \geq b) = \P(|Q^\e_{t,r}| \geq b)$, i.e., the distribution of $N_{t,r}$ coincides with that of $|Q^\e_{t,r}|$. Furthermore, since $Q^\e_{t,s}=\int_t^s ( 1 - \tanh(\dfrac{u}{\e})) dW_u$ it is normally distributed,
$$Q^\e_{t,s} \sim N(0, \int _t^s  ( 1 - \tanh(\dfrac{u}{\e}))^2 du),$$
and $|Q^\e_{t,s}|$ follows the corresponding half-normal distribution. Hence
$$\E [N^\e_{t,r}] = \E[|Q^\e_{t,r}|] = \l( \frac{2}{\pi} \int _t^r  ( 1 - \tanh(\dfrac{u}{\e})) ^2 du\r)^{1/2}
$$
and
\begin{align}\label{eq:Ninfty}
\E [N_{t,\infty} ^\e] =& \lim _{r\to \infty} \E [N_{t,r} ^\e] = \lim _{r\to \infty} \l ( \frac{2}{\pi} \int _t 
^r \l(1- \tanh(\dfrac{u}{\e}) \r)^2 du \r)^{1/2} \notag \\
\leq & \lim _{r\to \infty}  \l ( \frac{2}{\pi} \int _0 
^r \l(1- \tanh(\dfrac{u}{\e}) \r)^2 du \r)^{1/2} = \sqrt{ \e}\l( \frac{4}{\pi} \log (2) -\frac{2}{\pi} \r) ^{1/2}.
\end{align}
The result now follows by combining \eqref{eq:reads} and \eqref{eq:Ninfty}. \qed

\section{Numerical example} \label{sec:numerical}
We conclude with a numerical calculation showing some features stemming from the lack of information. In particular, we solve \eqref{eq:system} using the Crank-Nicolson finite difference scheme with linear interpolation at the boundaries. 
The value function $v_i^{\e}$ is found 
using the parameters in Table \ref{table:parameters}\footnote{The payoff function $\psi_1$ is chosen as $\psi_1(x)=10x$ to get a convenient order of magnitude of the value function.}. Recall that the value function $v^\e_i(t,m)$ is given w.r.t. $W^x_0=x$. For ease of exposition we only present numerical results for $v_1^\e(t,m)$ subject to the initial condition $W^x_0=0$. 

The monotonicity proved in Proposition~\ref{prop:monotone} is clearly seen in Figure~\ref{fig:results} and Table~\ref{tab:results}. When the  noise in the observation grows bigger it becomes less and less valuable and the value function tends towards the case of no information, i.e., $\lim_{\e \to \infty} v_1^\e = v_1^\infty(t,m)$ where $v_1^\infty(t,m) = 0$ for $x \leq 0$ and $v^\infty_1(t,m)=(T-t) \psi_1(m) = 10(T-t) m$ for $m> 0$. 

\begin{table}
\begin{center}

    \begin{tabular}{ | l| l | r |}
    \hline
\textbf{Description} & \textbf{Symbol} &\textbf{Value} \\ \hline
Cost of opening&   $c_{01}$  & 0.01  \\ \hline
Cost of closing & $c_{10}$ & $0.001$  \\   \hline
Running payoff state 1&$\psi_1(x)$&  $10x$  \\ \hline
 Running payoff state 0& $\psi_0(x)$ & 0  \\ \hline
Terminal time & T& 1 \\
\hline
    \end{tabular}
\end{center}
\caption{Parameter values}
\label{table:parameters}
\end{table}

\begin{table}
\begin{center}
    \begin{tabular}{|l |r|r|r|}
    \hline
\textbf{$v_1^\e(t,m)$} & $m= -0.5$ & $m=0$ & $m=0.5$ \\ \hline
$t=0$ &0.7680/0.0631/-0.001&2.2860/0.7898/0.0575&5.6481/5.0367/5.0000 \\ \hline
$t=0.5$&0.1814/0.0349/-0.001&0.8567/0.5069/0.0396&2.6015/2.5097/2.5000\\ \hline
    \end{tabular}
\end{center}

\caption{Numerical results for $\e=2^{-4} / 1 / 2^3$.}
\label{tab:results}
\end{table}

%

\setcounter{figure}{-2}
\begin{figure}
\centering
\begin{minipage}{.5\textwidth}
  \centering
    \includegraphics[width=1.0\linewidth]{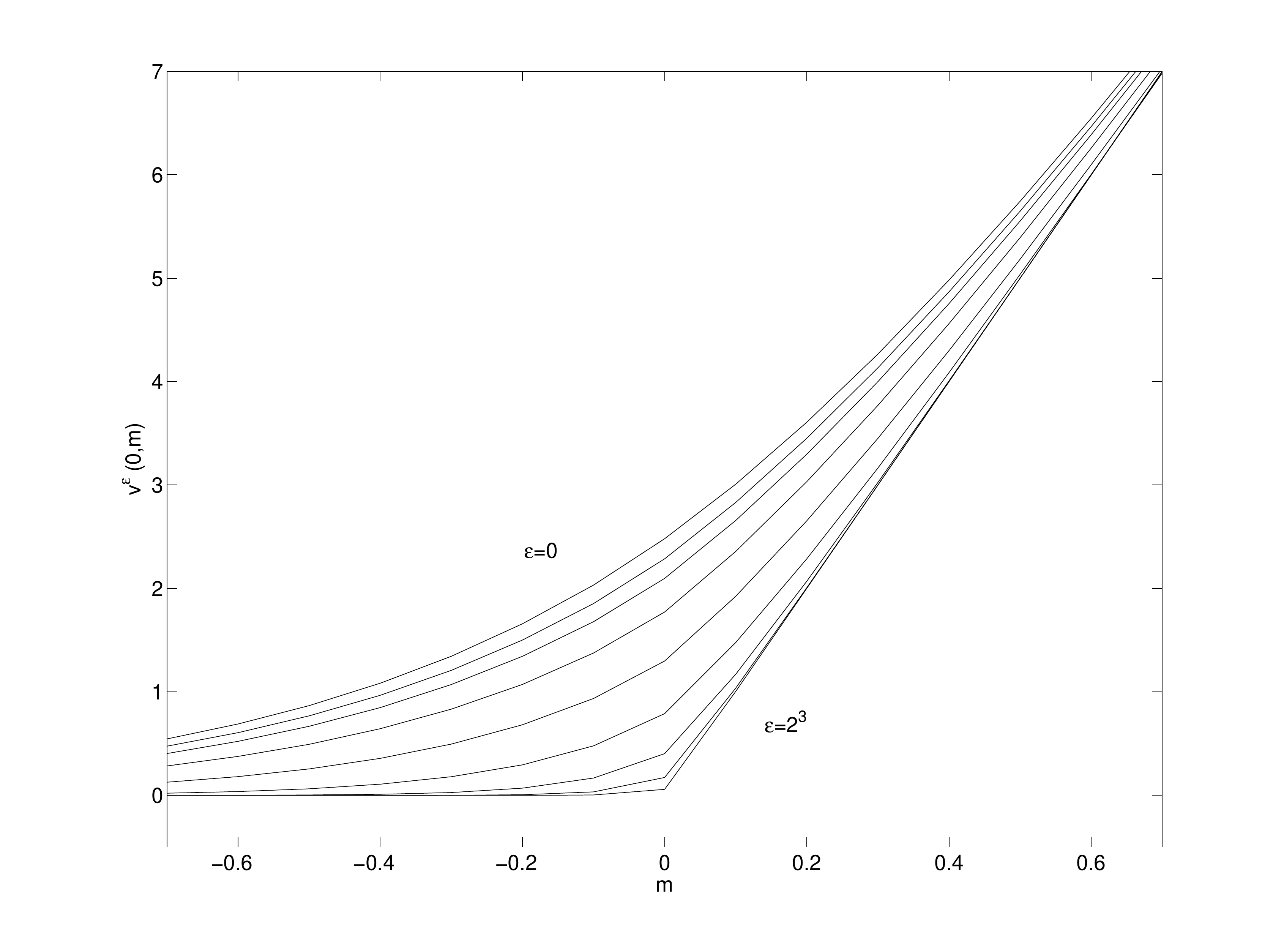}
      \caption{figure}{$(a)~t=0$}
\end{minipage}%
\begin{minipage}{.5\textwidth}
\centering
 \includegraphics[width=1.0\linewidth]{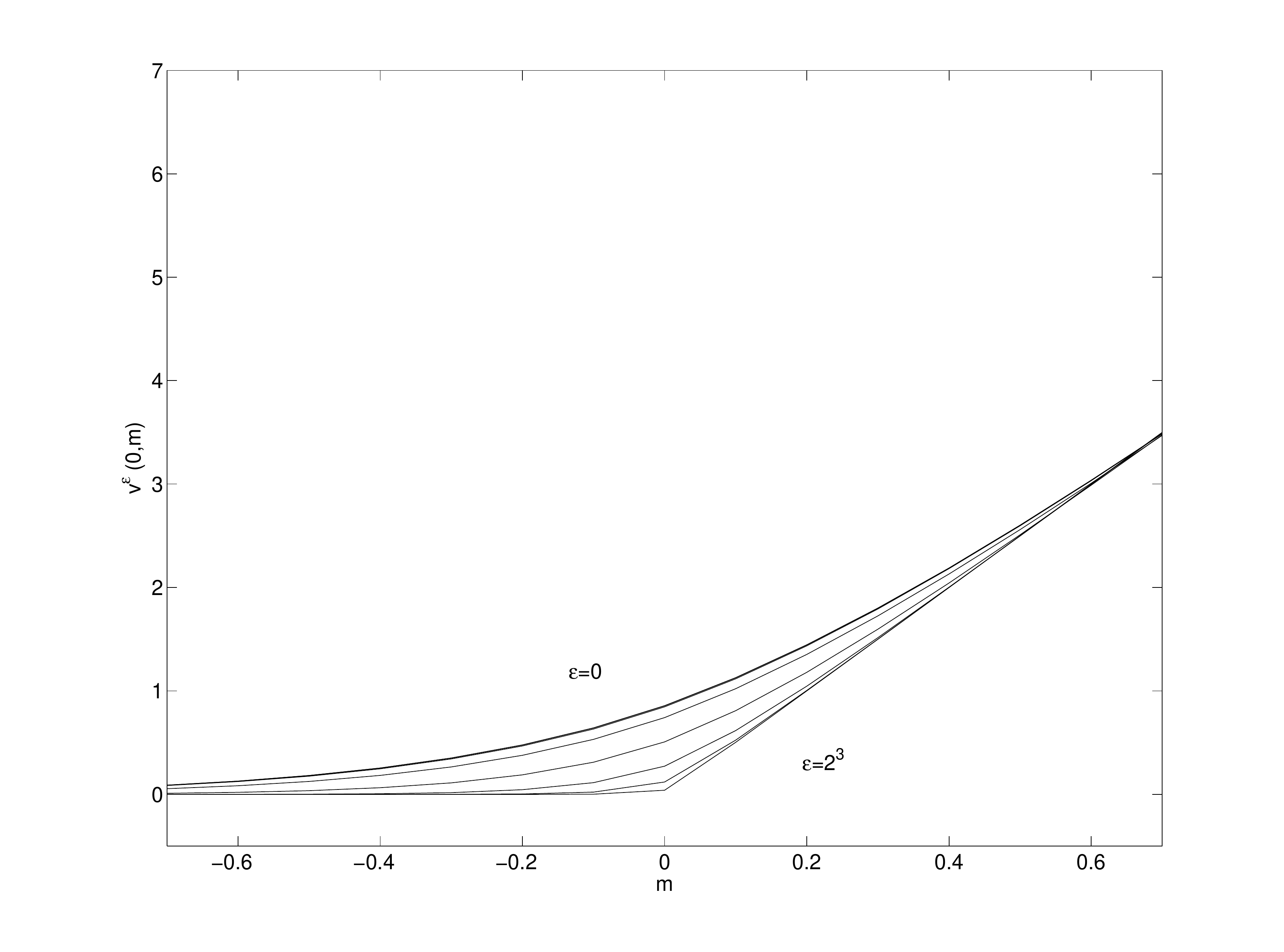}
    \caption{figure}{$(b)~t=0.5$}
\end{minipage}
\caption{$v^\e_1(t,m)$ for $\e \in \{0, 2^{-4}, \dots, 2^3 \}$.}
\label{fig:results}
\end{figure}

\section{Summary and conclusion} \label{sec:conclusions}

In this paper we studied a Brownian optimal switching problem under incomplete information. We showed that the value of information is positive and that the value function converges to the corresponding full information value function when the noise in the observation tends to $0$. Although the problem studied here is simplistic, the results indicate that the method of reducing to a full information problem, a technique successfully used in the study of incomplete information optimal stopping, provides a feasible way of tackling IIOSPs. An interesting and natural continuation of this paper is to study the IIOSP for more general stochastic processes and payoff functions/switching costs, firstly in the general setting of linear Kalman-Bucy filters and ultimately for fully non-linear stochastic filters. 
%
%
%
%
%
%


\begin{thebibliography}{99}



\bibitem[AH09]{AH09}
B. El-Asri and S. Hamadene,
\emph{The finite horizon optimal multi-modes switching problem: The viscosity solution approach}
Applied Mathematics \& Optimization, {\bf 60}(2009), 213-235.

\bibitem[ALL97]{ALL97}
O. Alvarez, J.-M. Lasry and P.-L. Lions,
\emph{Convex viscosity solutions and state constraints}
Journal de math{\'e}matiques pures et appliqu{\'e}es, {\bf 76}(1997), 265-288.

\bibitem[BC09]{BC09}
A. Bain, D. Crisan,
\textit{Fundamentals of stochastic filtering},
Stochastic Modelling and Applied Probability, vol. 60, Springer, 2009.

\bibitem[CIL92]{CIL92}
M. G. Crandall, H. Ishii, P-L Lions,
\emph{User's Guide to Viscosity Solutions of Second Order Partial Differential Equations},
Bulletin of the American Mathematical Society,  {\bf 27} (1992), 1-67.


\bibitem[DHP10]{DHP10}
B. Djehiche, S. Hamadene, A. Popier,
\emph{A Finite Horizon Optimal Multiple Switching Problem},
SIAM Journal on Control and Optimization, {\bf 48} (2010),  2751-2770

\bibitem[EKPPQ97]{EKPPQ97}
El-Karoui, Kampoudjian, Pardoux, Peng, Quenez
\emph{Reflected Solutions of Backward SDE's and Related Obstacle Problems for PDE's},
The Annals of Probability, {\bf 25} (1997), 702-737.

\bibitem[F78]{F78}
H. F\"ahrmann,
\emph{On optimal stopping of a Wiener process with incomplete data},
Theory of Probability \& its Applications, {\bf 23} (1978), 138-143.


\bibitem[HM12]{HM12}
S. Hamad\'ene, M.A. Morlais
\emph{Viscosity Solutions of Systems of PDEs with Interconnected Obstacles and Multi-Modes Switching Problem},
arXiv:1104.2689v2.

\bibitem[HT07]{HT07}
Y. Hu, S. Tang
\emph{Multi-dimensional BSDE with Oblique Reflection and Optimal Switching},
Probability Theory and Related Fields, {\bf 147} (2010), 89-121.


\bibitem[HZ10]{HZ10}
S. Hamad\'ene, J. Zhang
\emph{Switching Problem and Related System of Reflected Backward SDEs},
Stochastic Processes and their Applications, {\bf 120} (2010), 403-426.



\bibitem[LiNO14]{LiNO14a}
K. Li, K. Nystr\"om, M. Olofsson,
\emph{Optimal switching problems under incomplete information},
To appear in Monte Carlo Methods and Applications.
%

\bibitem[LuNO12]{LuNO12} N.L.P. Lundstr\"{o}m, K. Nystr{\"o}m, M. Olofsson,
\emph{Systems of variational inequalities in the context of Optimal Switching Problems and Operators of Kolmogorov Type},
Annali di Mathematica Pura ed Applicata, {\bf 193} (2014), 1213-1247.

\bibitem[LS78]{LS78} R.S. Lipster, A.N. Shiryayev,
\emph{Statistics of random processes},
 (1978), Springer-Verlag, New-York.

\bibitem[TY93]{TY93}
S. Tang, J. Yong,
\emph{Finite horizon stochastic optimal switching and impulse controls with a viscosity solution approach},
Stochastics and Stochastics Reports, {\bf 45}(1993), 145-176.

%

\end{thebibliography}
\end{document}